\documentclass[11pt,a4paper]{article}

\usepackage[left=3cm, top=3cm,bottom=3cm,right=3cm]{geometry}
\usepackage{mathtools,amssymb,amsthm,mathrsfs,calc,graphicx,xcolor,cleveref,dsfont,tikz,pgfplots,bm, bbm}
\usepackage{tikz-cd}

\usepackage[british]{babel}
\usepackage{amsfonts}              
\usepackage[T1]{fontenc}

\newtheorem {theorem}{Theorem}

\newtheorem {corollary}[theorem]{Corollary}

{\theoremstyle{definition}

}

{
\theoremstyle{remark}
\newtheorem {remark}[theorem]{Remark}
}

\newcommand{\var}{\operatorname{Var}}

\def\CC{\mathbb{C}}

\def\EE{\mathbb{E}}

\def\NN{\mathbb{N}}

\def\PP{\mathbb{P}}

\def\RR{\mathbb{R}}
\def\RRd1{\mathbb{R}^{d+1}}

\def\cN{\mathcal{N}}

\def\dint{\textup{d}}

\makeatletter
\let\@fnsymbol\@alph
\makeatother

\begin{document}

\title{\bfseries On random convex chains, orthogonal polynomials, PF sequences and probabilistic limit theorems}

\author{Anna Gusakova\footnotemark[1]\;\;\;\;\;\; Christoph Th\"ale\footnotemark[2]}

\date{}
\renewcommand{\thefootnote}{\fnsymbol{footnote}}

\footnotetext[1]{Fakult\"at f\"ur Mathematik, Ruhr-Universit\"at Bochum, Germany. Email: anna.gusakova@rub.de}

\footnotetext[2]{Fakult\"at f\"ur Mathematik, Ruhr-Universit\"at Bochum, Germany. Email: christoph.thaele@rub.de}

\maketitle

\begin{abstract}
\noindent  Let $T$ be the triangle in the plane with vertices $(0,0)$, $(0,1)$ and $(0,1)$. The convex hull of $(0,1)$, $(1,0)$ and $n$ independent random points uniformly distributed in $T$ is the random convex chain $T_n$. A three-term recursion for the probability generating function $G_n$ of the number $f_0(T_n)$ of vertices of $T_n$ is proved. Via the link to orthogonal polynomials it is shown that $G_n$ has precisely $n$ distinct real roots in $(-\infty,0]$ and that the sequence $p_k^{(n)}:=\mathbb{P}(f_0(T_n)=k)$, $k=1,\ldots,n$, is a Polya frequency (PF) sequence. A selection of probabilistic consequences of this surprising and remarkable fact are discussed in detail.
\end{abstract}

\section{Introduction}

\begin{figure}[t]
\begin{center}
	\includegraphics[width=0.4\columnwidth]{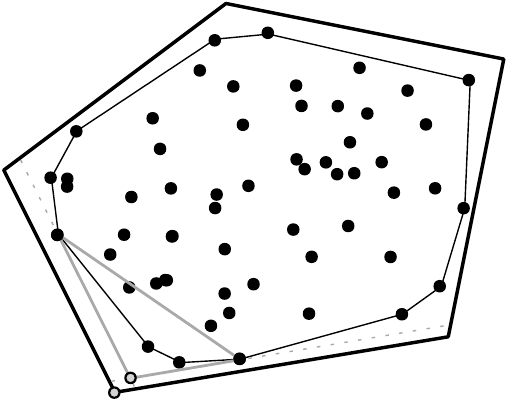}\qquad\qquad\qquad\qquad
	\includegraphics[width=0.3\columnwidth]{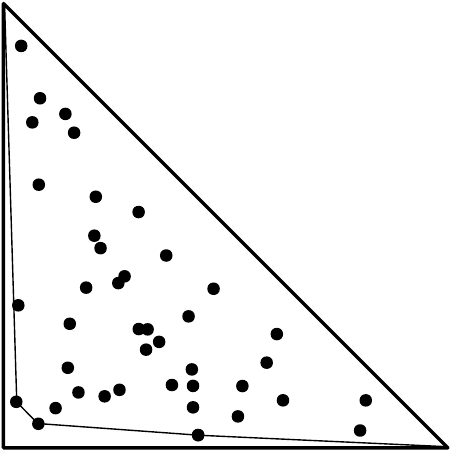}
\end{center}
\label{fig:poltri}
\caption{Illustration of a random polygon $P_n$ in a polygon $P$ together with a marked corner illustrating the reduction to a random convex chain (left) and a random convex chain $T_n$ in the right triangle $T$ (right).}
\end{figure}

Let $P\subset\RR^2$ be a convex polygon with $v\geq 3$ vertices. Further, let $X_1,\ldots,X_n$ be $n\in\NN=\{1,2,\ldots\}$ random points sampled independently and uniformly in $P$. Their convex hull $P_n$ is a random polygon in $P$, see Figure \ref{fig:poltri}. Such random polygons and their higher dimensional version are central objects studied in convex geometry and geometric probability. From the pioneering work of R\'enyi and Sulanke \cite{RenyiSulanke} it is known that the expected number of vertices $\EE f_0(P_n)$ satisfies
$$
\EE f_0(P_n) = \Big({2\over 3}\,v\,\log n\Big) (1+o(1)),\qquad\text{as}\qquad n\to\infty.
$$
After this result, despite many efforts, it took around 25 years until Groeneboom \cite{Groeneboom88} was able to compute the asymptotic variance of $f_0(P_n)$:
$$
\var f_0(P_n) = \Big({10\over 27}\,v\,\log n\Big) (1+o(1)),\qquad\text{as}\qquad n\to\infty.
$$
Moreover, he proved that the $f_0(P_n)$ satisfies a central limit theorem, that is,
$$
{f_0(P_n)- \EE f_0(P_n)\over\sqrt{\var f_0(P_n)}} \overset{d}{\longrightarrow} Z,\qquad\text{as}\qquad n\to\infty,
$$
where $\overset{d}{\longrightarrow}$ denotes convergence in distribution and $Z\sim\cN(0,1)$ stands for a standard Gaussian random variable, see also \cite{Groeneboom12,NagaevKhamdamov,Pardon}.

The argument of Groeneboom \cite{Groeneboom88} shows that the vertices of $P_n$ accumulate close to the corners of the polygon $P$ and that the behaviour in different corners is essentially independent. This motivates the study of $P_n$ in the neighbourhood of a single corner, which by affine invariance of $f_0(P_n)$ can be identified with the apex of a right angle. Therefore, if $T$ denotes the triangle with corners at $(0,0)$, $(0,1)$ and $(1,0)$ one considers the convex chain $T_n$, which is the convex hull of $n\in\NN$ uniform random points in $T$ together with the corners $(0,1)$ and $(1,0)$, see Figure \ref{fig:poltri}. 
In this reduced situation Buchta \cite{Buchta06} was able to determine the precise distribution of $f_0(T_n)$. He shows that, for $k=1,\ldots,n$,
\begin{equation}\label{eq:Probabilities}
\begin{split}
p_n^{(k)}:=\PP(f_0(T_n)=k) &= 2^k\sum_{i_1+\ldots+i_k=n}{1\over i_1(i_1+i_2)\cdots(i_1+\ldots+i_k)}\\
&\hspace{3cm}\times{i_1\cdots i_k\over (i_1+1)(i_1+i_2+1)\cdots(i_1+\ldots+i_k+1)},
\end{split}
\end{equation}
where the sum runs over all integers $i_1,\ldots,i_k\geq 1$ satisfying $i_1+\ldots+i_k=n$. We remark that the value $p_n^{(n)}={2^n\over n!(n+1)!}$ was earlier determined by B\'ar\'any, Roter Steiger and Zhang \cite{BaranyRoteSteigerZhang} (only a discrepancy of $2^n$ arises here since in \cite{BaranyRoteSteigerZhang} the square with vertices $(0,0)$, $(0,1)$, $(1,0)$ and $(0,0)$ instead of the triangle $T$ was considered). Although these explicit expressions are very hard to work with, in a subsequent paper Buchta \cite{Buchta12} was able to compute the first few moments of $f_0(P_n)$ explicitly and to derive a recursion scheme which, in principle, allows the computation of all higher moments as well. For example,
\begin{align*}
\EE f_0(T_n) &= {2\over 3}\sum_{k=1}^n{1\over k} + {1\over 3},\\
\EE f_0(T_n)^2 &= {4\over 9}\Big(\sum_{k=1}^n{1\over k}\Big)^2 + {22\over 27}\sum_{k=1}^n{1\over k}+{4\over 9}\sum_{k=1}^n{1\over k^2}-{25\over 27}+{4\over 9(n+1)}.
\end{align*}
However, the exact formulas (even for moments of order $3$ or $4$) become more and more involved so that we only mention the asymptotic relation, valid for $m\in\NN$,
\begin{align*}
	\EE f_0(T_n)^m = \Big({2\over 3}\log n\Big)^m + O((\log n)^{m-1}),\qquad\text{as}\qquad n\to\infty,
\end{align*}
which was derived in \cite{Buchta12} as well. 

The present paper takes up this classical topic once again. Our goal is to study fine probabilistic properties of the random variables $f_0(T_n)$. To this end, we develop further Buchta's approach from \cite{Buchta12} and deduce first a three-term recursion for the probability generating functions of $f_0(T_n)$. This establishes a new link between the number of vertices of the random convex chains $T_n$ and the theory of orthogonal polynomials. We argue that, very surprisingly and remarkably, all zeros of these generating functions are real, implying that for any fixed $n\in\NN$, the probabilities $(p_n^{(k)})_{k=0}^n$ form a so-called P\'olya frequency (PF) sequence.  Discrete probability distributions generated by such sequences are well studied and a large number of probabilistic results, such as central or local limit theorems and large deviation estimates, are available for them. This approach to probabilistic limit theorems dates back to the work of Harper \cite{Harper} and is rather well established in analytic combinatorics, since it applies to many random variables arising in this area. Examples include statistics associated with random permutations, random trees, random maps, random matchings or random graphs, to name just a few. The paper of Pitman \cite{Pitman} gives an impressive overview with many examples and a large number of references. However, we are not aware of any problem in geometric probability, for which this approach has so far been implemented. Our paper does this for the number of vertices of the random convex chains $T_n$ and this way we are able to shed new light onto probabilistic estimates for these random variables, which were not available before. It is our hope that this note will stimulate further research in this potentially fruitful direction in geometric probability and stochastic geometry.

\section{Generating function of the vertex number}

\subsection{Recurrence relations}

Let us briefly recall out set up: $X_1,\ldots,X_n$ are $n\in\NN$ independent random points uniformly distributed in the triangle $T$ with vertices $(0,0)$, $(0,1)$ and $(1,0)$ and $T_n$ is the convex chain generated by $X_1,\ldots,X_n$, that is, the convex hull of $X_1,\ldots,X_n$ together with the two vertices $(0,1)$ and $(1,0)$ of $T$. The random variable of interest is the number $f_0(T_n)$ of vertices of $T_n$ and we put $p_n^{(k)}:=\PP(f_0(T_n)=k)$, $k=1,\ldots,n$. The probability generating function of $f_0(T_n)$ is defined as
$$
G_n(z) := \sum_{k=1}^np_n^{(k)}z^k,\qquad z\in\RR.
$$
Note that $G_1(z)=z$ and for convenience we also define $G_0(z)=1$. Our first result is a three-term recurrence for the functions $G_n(z)$.

\begin{theorem}\label{thm:pgf}
The probability generating functions $G_n(z)$ of the random variables $f_0(T_n)$ satisfy the three-term recurrence relation
$$
G_n(z) = (a_nz+b_n)G_{n-1}(z)-c_nG_{n-2}(z),\qquad n\geq 2,\qquad z\in\RR,
$$
with
$$
a_n := {2\over n(n+1)},\qquad b_n:=2\,{n-1\over n+1}\qquad\text{and}\qquad c_n:={(n-1)(n-2)\over n(n+1)},
$$
and with initial data $G_0(z)=1$ and $G_1(z)=z$.
\end{theorem}
\begin{proof}
The argument is based on the explicit representation of the probabilities $p_k^{(n)}$ from \eqref{eq:Probabilities}. We start by recalling an argument from the proof of Theorem 1 in \cite{Buchta12}. Since $i_1,\ldots,i_k\geq 1$ in \eqref{eq:Probabilities}, we necessarily have that $1\leq i_k\leq n-k+1$ and $k-1\leq j:=i_1+\ldots+i_{k-1}\leq n-1$. It follows that $p_k^{(n)}$ can be rewritten as
$$
p_k^{(n)} = {2\over n(n+1)}\sum_{j=k-1}^{n-1}(n-j)p_{k-1}^{(j)}
$$
with $p_0^{(0)}=1$ and $p_0^{(j)}=0$ for $j\geq 1$. Multiplying the last equality with $n(n+1)\over 2$ and doing the same with $p_k^{(n-1)}$ we conclude that
\begin{equation}\label{eq:rec0}
{n(n+1)\over 2}p_k^{(n)} - {n(n-1)\over 2}p_k^{(n-1)} = \sum_{j=k-1}^{n-1}p_{k-1}^{(j)}.
\end{equation}
Replacing $n$ by $n-1$ and subtracting the resulting identity from \eqref{eq:rec0}, we see that
\begin{align*}
\Big({n(n+1)\over 2}p_k^{(n)} - {n(n-1)\over 2}p_k^{(n-1)}\Big) - \Big({n(n-1)\over 2}p_k^{(n-1)} - {(n-1)(n-2)\over 2}p_k^{(n-2)}\Big) = p_{k-1}^{(n-1)},
\end{align*}
or, equivalently,
\begin{equation}\label{eq:rec2}
{n(n+1)\over 2}p_k^{(n)} - n(n-1)p_k^{(n-1)} + {(n-1)(n-2)\over 2}p_k^{(n-2)} = p_{k-1}^{(n-1)}
\end{equation}
for $n\geq 2$ and $k=1,\ldots,n$.

We now multiply both sides of \eqref{eq:rec2} by $z^k$, $z\in\RR$ and sum over all possible values of $k$:
\begin{align*}
{n(n+1)\over 2}\sum_{k=1}^np_k^{(n)}z^k - n(n-1)\sum_{k=1}^{n-1}p_k^{(n-1)}z^k + {(n-1)(n-2)\over 2}\sum_{k=1}^{n-2}p_k^{(n-2)}z^k = \sum_{k=2}^{n}p_{k-1}^{(n-1)}z^k.
\end{align*}
Using the definition of the probability generating function of $f_0(T_n)$ this can be rephrased by saying that
\begin{align*}
{n(n+1)\over 2}G_n(z) - n(n-1)G_{n-1}(z) + {(n-1)(n-2)\over 2}G_{n-2}(z) = zG_{n-1}(z).
\end{align*}
Solving for $G_n(z)$ leads to
\begin{align*}
G_n(z) &= {2\over n(n+1)}\Big[zG_{n-1}(z)+n(n-1)G_{n-1}(z)-{(n-1)(n-2)\over 2}G_{n-2}(z)\Big]\\
&=\Big({2\over n(n+1)}\,z+2{n-1\over n+1}\Big)G_{n-1}(z) - {(n-1)(n-2)\over n(n+1)}G_{n-2}(z)\\
&= (a_nz+b_n)G_{n-1}(z) - c_nG_{n-2}(z).
\end{align*}
This completes the argument.
\end{proof}

The previous theorem can be used to determine recursively the functions $G_n(z)$. For example,
\begin{align*}
	G_2(z) &= {1\over 3}\,z(z+2)\\
	G_3(z) &= {1\over 18}\,z(z^2+8z+9)\\
	G_4(z) &= {1\over 180}\,z(z^3+20z^2+87z+72)\\
	G_5(z) &= {1\over 2\,700}\,z(z^4+40z^3+427z^2+1\,332z+900)\\
	G_6(z) &= {1\over 56\,700}\,z(z^5+70z^4+1\,477z^3+11\,142z^2+27\,810z+16\,200),
\end{align*}
see Figure \ref{fig:poltri2}.

\begin{figure}[t]

\begin{center}
	\includegraphics[width=0.45\columnwidth]{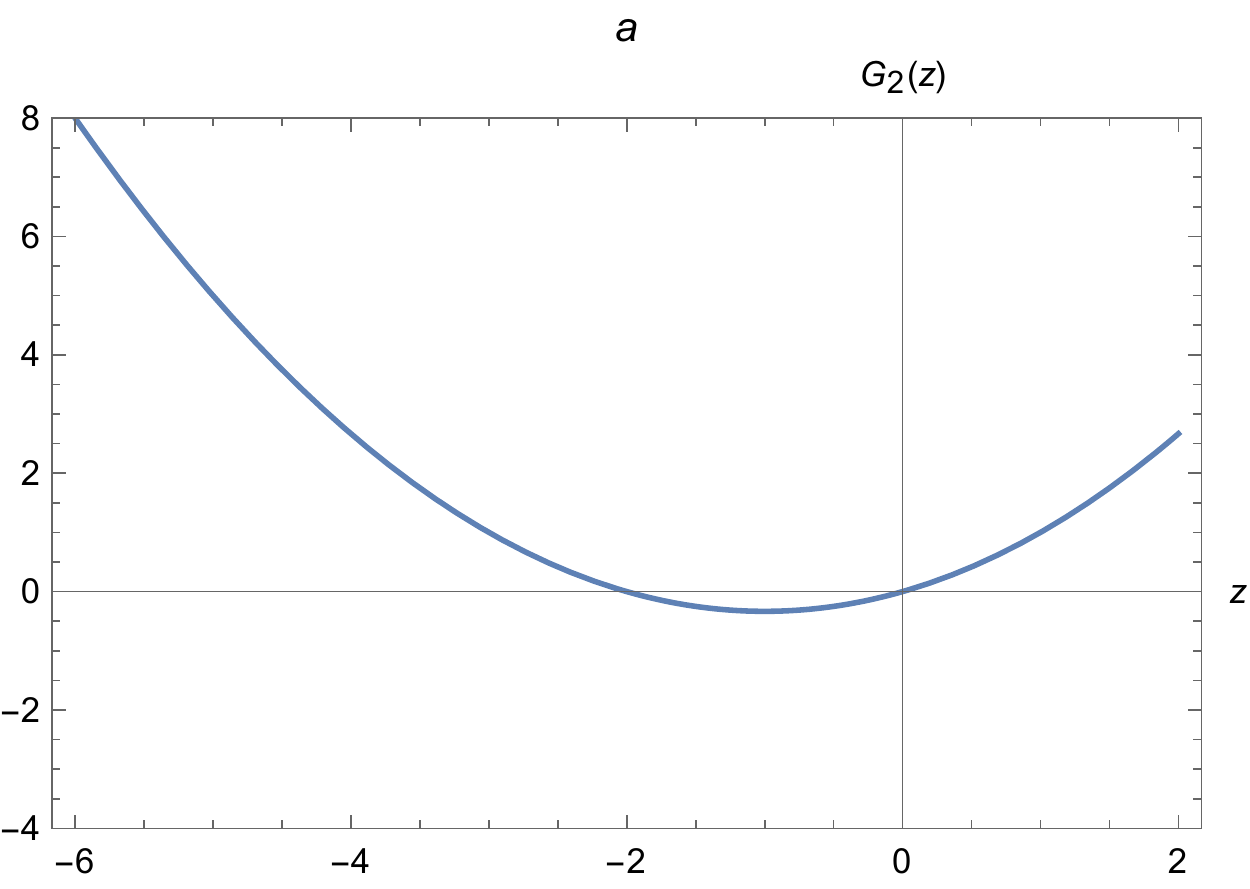}\qquad
	\includegraphics[width=0.45\columnwidth]{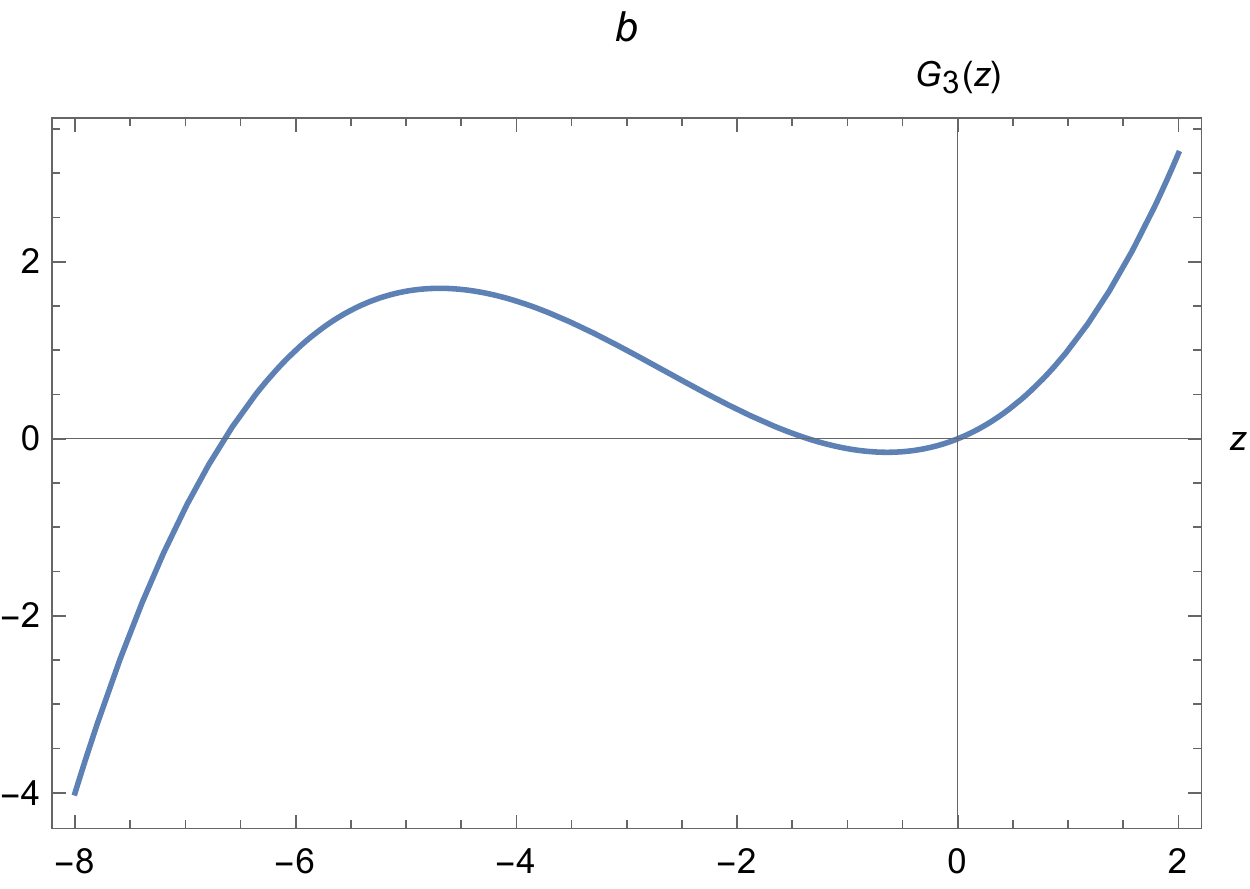}\\
	\includegraphics[width=0.45\columnwidth]{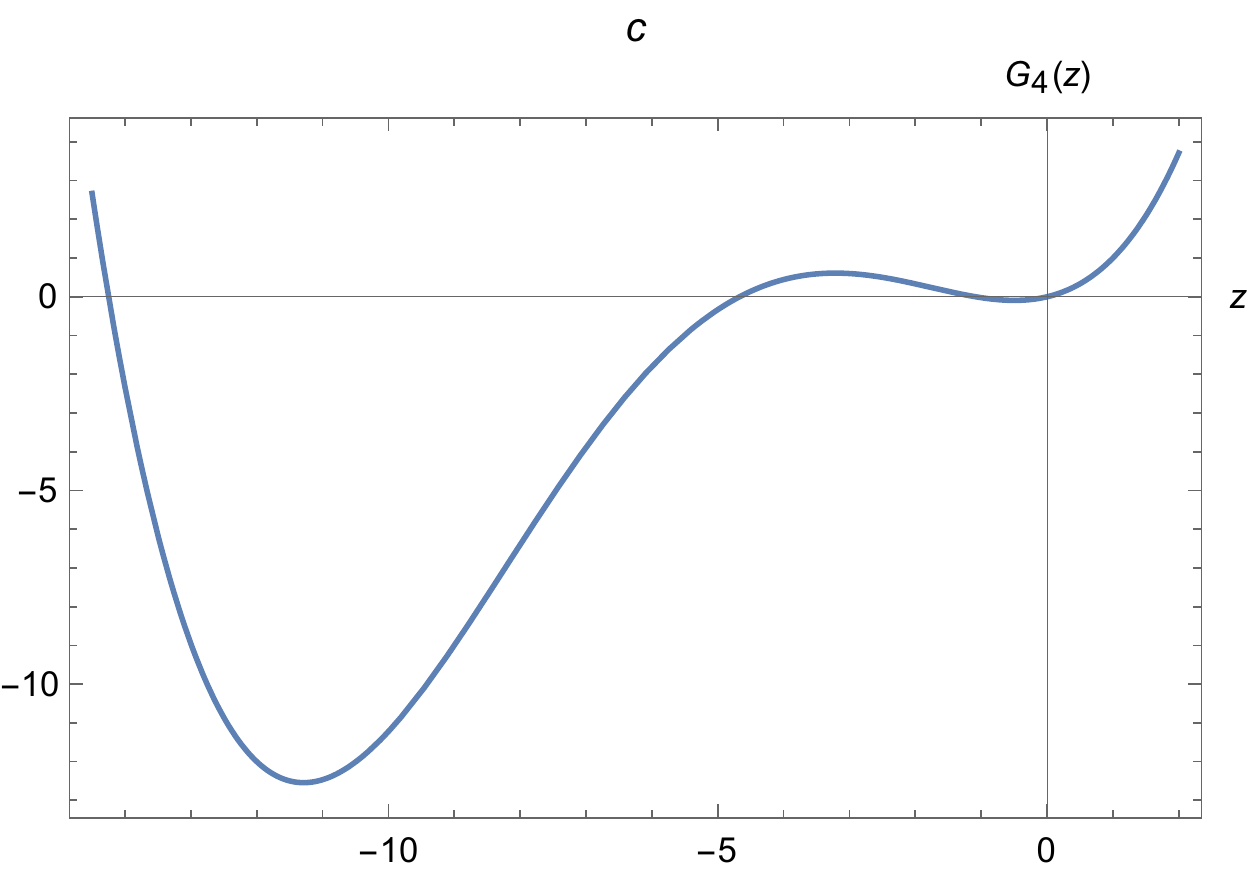}\qquad
	\includegraphics[width=0.45\columnwidth]{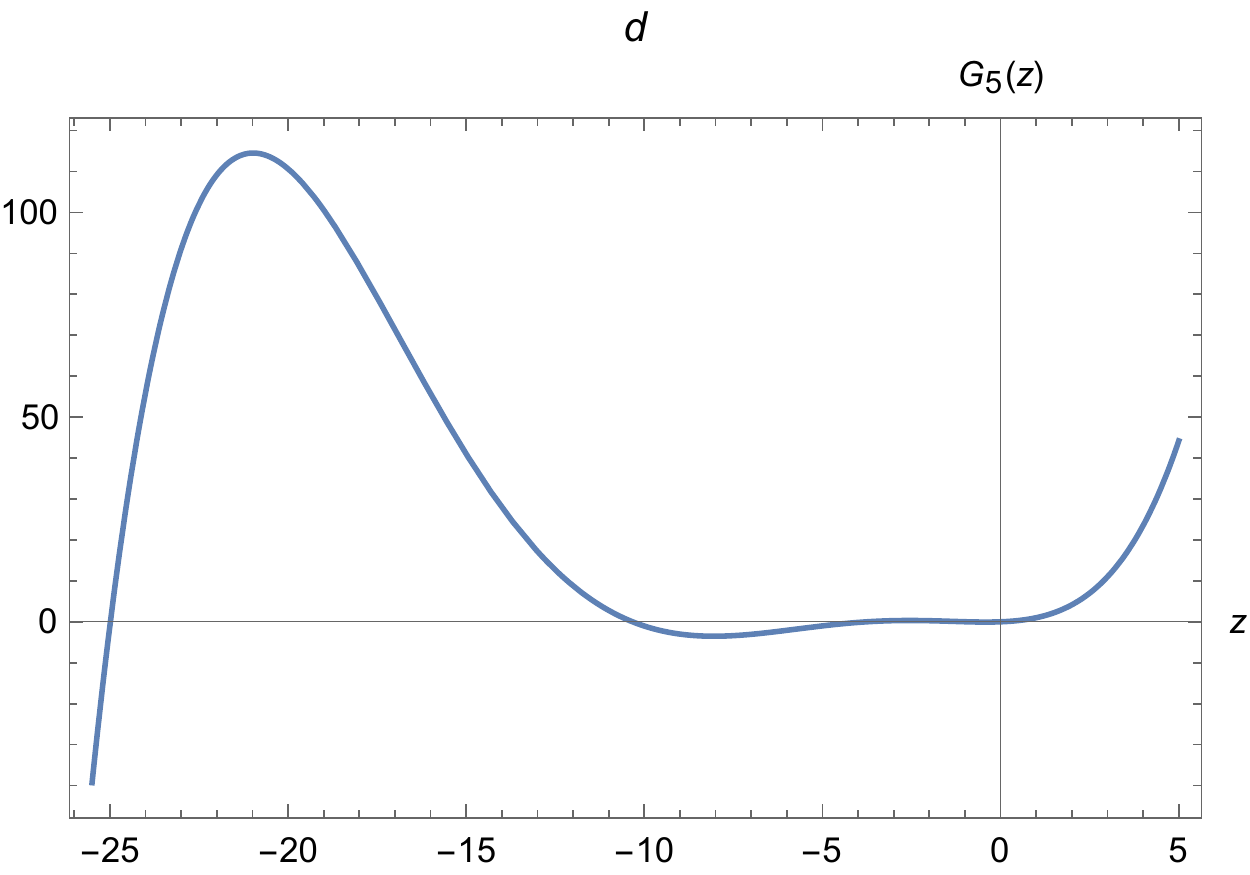}
\end{center}
\caption{Graphs of the functions $G_n(z)$ illustrating the location of the roots: (a) $n=2$, (b) $n=3$, (c) $n=4$, (d) $n=5$.}
	\label{fig:poltri2}
\end{figure}

Alternatively to the three-term recursion we also have the following representation of $G_n(z)$, which is close in spirit to the representation of the $m$-th moment $\EE f_0(T_n)^m$ in \cite[Theorem 2]{Buchta12}. In particular, the same weights $w_{k,n}$ are involved.

\begin{theorem}\label{thm:pgfAlternative}
For $n\geq 2$ the probability generating function $G_n(z)$ of $f_0(T_n)$ satisfies
$$
G_n(z) = z +(z-1)\sum_{k=1}^{n-1}w_{k,n}G_k(z),\qquad z\in\RR,
$$
with initial data $G_1(z)=z$ and with the weights $w_{k,n}$ given by
$$
w_{k,n}=2k(k+1)\sum_{i=k+1}^n{1\over i^2(i^2-1)}.
$$
\end{theorem}
\begin{proof}
Let $f,g:\NN\to\RR$ be arbitrary functions satisfying, for $n\geq 2$, the difference equation
\begin{equation}\label{eq:DifferenceEquation}
{n(n+1)\over 2}f(n)-(n^2-n+1)f(n-1)+{(n-1)(n-2)\over 2}f(n-2) = g(n).
\end{equation}
In the proof of Theorem 2 in \cite{Buchta12} it has been shown that in such a situation $f(n)$ is given by
$$
f(n) = f(1) + \sum_{i=2}^n{2\over i^2(i^2-1)}\Big(6(f(2)-f(1))+\sum_{k=3}^i (k-1)kg(k)\Big).
$$
By subtracting $G_{n-1}(z)$ from both sides in the recurrence equation in Theorem \ref{thm:pgf} we see that for any $z\in\RR$ the function $f(n)=G_n(z)$ satisfies \eqref{eq:DifferenceEquation} with $g(n)=(z-1)G_{n-1}(z)$. Thus,
\begin{align*}
	G_n(z) = G_1(z) + \sum_{i=2}^n{2\over i^2(i^2-1)}\Big(6(G_2(z)-G_1(z))+\sum_{k=3}^i (k-1)k(z-1)G_{k-1}(z)\Big).
\end{align*}
Since $G_1(z)=z$ and $G_2(z)={1\over 3}z(z+2)$ this reduces to
\begin{align*}
G_n(z) &= z + (z-1)\sum_{i=2}^n{2\over i^2(i^2-1)}\Big(2z+\sum_{k=3}^i (k-1)kG_{k-1}(z)\Big)\\
&= z + (z-1)\sum_{i=2}^n{2\over i^2(i^2-1)}\sum_{k=2}^i (k-1)kG_{k-1}(z)\\
&= z + (z-1)\sum_{i=2}^n{2\over i^2(i^2-1)}\sum_{k=1}^{i-1} k(k+1)G_{k}(z).
\end{align*}
Interchanging both sums and using the definition of $w_{k,n}$ leads to
\begin{align*}
G_n(z) &= z + (z-1)\sum_{k=1}^{n-1}k(k+1)G_k(z)\sum_{i=k+1}^n{2\over i^2(i^2-1)}\\
&=z +(z-1)\sum_{k=1}^{n-1}w_{k,n}G_k(z).
\end{align*}
This completes the proof.
\end{proof}

\subsection{Associated orthogonal polynomials and PF sequence property}

It directly follows from Theorem \ref{thm:pgf} or Theorem \ref{thm:pgfAlternative} that for each $n\in\NN$, $G_n(z)$ is a polynomial of degree $n$ with leading term $p_n^{(n)}z^n$. Consequently, these polynomials are not monic, meaning that their leading coefficients are different from $1$ (except if $n=1$, of course). On the other hand, dividing them by $p_n^{(n)}={2^n\over n!(n+1)!}$ we obtain the sequence of monic polynomials
$$
H_n(z) := {G_n(z)\over p_n^{(n)}},\qquad n\in\NN,\qquad z\in\RR.
$$
The poynomials $H_n(z)$ satisfy the following three-term recurrence relation.

\begin{theorem}
The polynomials $H_n(z)$ satisfy the three-term recurrence relation
$$
H_n(z) = (z+\beta_n)H_{n-1}(z) - \gamma_nH_n(z),\qquad n\geq 2,\qquad z\in\RR,
$$
with 
$$
\beta_n := n(n-1)\qquad\text{and}\qquad \gamma_n := {n(n-1)^2(n-2)\over 4},
$$
and with initial data $H_0(z)=1$ and $H_1(z)=z$.
\end{theorem}
\begin{proof}
Dividing both sides of the recurrence relation in Theorem \ref{thm:pgf} by $p_n^{(n)}$ leads to
\begin{align}\label{eq:rec1}
{G_n(n)\over p_n^{(n)}} = (a_nz+b_n)\,{G_{n-1}(z)\over p_n^{(n)}} - c_n\,{G_{n-2}(z)\over p_n^{(n)}}.
\end{align}
Using that $p_n^{(n)}={2^n\over n!(n+1)!}$ we have that, for $n\geq 2$,
$$
p_{n-1}^{(n-1)} = {2^{n-1}\over (n-1)!n!} = {n(n+1)\over 2}\,p_n^{(n)}
$$
and
$$
p_{n-2}^{(n-2)} = {2^{n-2}\over (n-2)!(n-1)!} = {n^2(n-1)(n+1)\over 4}\,p_n^{(n)}.
$$
Consequently, \eqref{eq:rec1} can be rewritten in terms of $H_n(z)$, $H_{n-1}(z)$ and $H_{n-2}(z)$ as
\begin{align*}
H_n(z) &= \Big({n(n+1)\over 2}a_n\,z+{n(n+1)\over 2}b_n\Big)H_{n-1}(z) - {n^2(n-1)(n+1)\over 4}c_n\,H_{n-2}(z)\\
&=\Big(z+n(n-1)\Big)H_{n-1}(z) - {n(n-1)^2(n-2)\over 4}\,H_{n-2}(z)\\
&=(z+\beta_n)H_{n-1}(z) - \gamma_n\,H_{n-2}(z),
\end{align*}
which completes the argument.
\end{proof}

Three-term recurrence relations as the one satisfied by $H_n(z)$ according to the previous theorem are very classical in mathematics and characterize orthogonal polynomials. In fact, Favard's theorem \cite{Favard,MarcellanAlvarez} implies that there exists a unique positive Borel measure $\mu$ on the real line such that the sequence of polynomials $(H_n(z))_{n\geq 0}$ is orthogonal with respect to the canonical inner product $\langle f,g\rangle_\mu := \int f(x)g(x)\,\mu(\dint x)$
induced by $\mu$. It seems a non-trivial task to construct the measure $\mu$. On the other hand, for our purposes the existence of $\mu$ is sufficient and it is not clear to us if the explicit knowledge of $\mu$ would have any further probabilistic consequences.

The zeros of orthogonal polynomials are well understood, see Chapter 3 in the classical reference \cite{Szego}. In particular, they are all real and distinct according to \cite[Theorem 3.3.1]{Szego}. Since the orthogonal polynomials $H_n(z)$ are just normalizations of the generating functions $G_n(z)$, the same is true for the zeros of $G_n(z)$ as well. In addition, since the coefficients of $G_n(z)$ are positive (they are probabilities), all roots of $G_n(z)$ are necessarily located on the non-positive real half-axis. We summarize our findings in the next theorem.

\begin{theorem}\label{thm:Zeros}
For each $n\in\NN$ the probability generating function $G_n(z)$ of $f_0(T_n)$ is a polynomial of degree $n$ with precisely $n$ distinct real roots in $(-\infty,0]$.
\end{theorem}

\begin{remark}\rm 
The result of Theorem \ref{thm:Zeros} can also be concluded from \cite[Corollary 2.4]{LuiWang}. In fact, from this it follows that the sequence of polynomials $G_n(z)$ is a so-called Sturmian sequence. For such sequence it is known that its members only have real roots. We preferred the more classical way via orthogonal polynomials.
\end{remark}

Next, we recall from \cite{Karlin} some classical notions from the theory of total positivity. An infinite matrix $A=(a_{ij})_{i,j\geq 0}$ is called totally positive if all minors of $A$ have non-negative determinant. An infinite sequence $(a_k)_{k\geq 1}$ of non-negative real numbers is a P\'olya frequence (PF) sequence if the matrix $(a_{i-j})_{i,j\geq 0}$ is totally positive, where one puts $a_k:=0$ if $k<0$. A finite sequence $(a_k)_{k=0}^n$ is a PF sequence, provided the infinite sequence $a_0,a_1,\ldots,a_n,0,0,\ldots$ is a PF sequence. PF sequences appear often in combinatorial problems and have remarkable properties. We refer to \cite{Pitman} for an overview containing many applications and references. For example, a finite PF sequence $(a_k)_{k=0}^n$ satisfies a strong form of a log-concavity property, meaning that
$$
a_k^2 \geq a_{k-1}a_{k+1}\Big(1+{1\over k}\Big)\Big(1+{1\over n-k}\Big)
$$
for $k=1,\ldots n-1$. Moreover, discrete probability distributions generated by PF sequences also display a large number of remarkable properties, some of which we discuss in the Section \ref{sec:Conseq} below. Our next result shows that the sequence of probabilities $(p_k^{(n)})_{k=1}^n$ is a PF sequence. To the best of our knowledge, this is the fist time, that such a property is established for a discrete probability distribution arising in the area of geometric probability. As just anticipated, we will discuss some consequences of this fact in the next section.

\begin{theorem}\label{thm:PF}
For any $n\in\NN$ the sequence $(p_k^{(n)})_{k=1}^n$ given by \eqref{eq:Probabilities} is a PF sequence.
\end{theorem}
\begin{proof}
This follows from Theorem \ref{thm:Zeros} together with the Aissen-Schoenberg-Whitney theorem from \cite{AissenEtAl}, which says that a finite sequence $(a_k)_{k=0}^n$ of non-negative real numbers if a PF sequence if and only if the generating function $\sum_{k=0}^na_kz^k$ has only real roots.
\end{proof}

\section{Probabilistic consequences for the vertex number}\label{sec:Conseq}

In this section we describe a selection of probabilistic consequences that Theorem \ref{thm:Zeros} and Theorem \ref{thm:PF} have for the number of vertices $f_0(T_n)$ of the random convex chain in the triangle $T$.

\paragraph{Representation as sum of independent random variables.}

We start our discussion with the following unexpected decomposition for which we do not have a geometric interpretation. Namely, according to Theorem \ref{thm:Zeros} the probability generating function $G_n(z)$ of $f_0(T_n)$ has precisely $n$ distinct roots in $(-\infty,0]$, which we denote by $-\infty<r_n<\ldots<r_2<r_1\leq 0$. Since $p_0^{(n)}=0$ for $n\geq 1$, $z$ is a trivial factor of $G_n(z)$, so that $r_1=0$. As a consequence, $G_n(z)$ can be decomposed into a product of linear factors:
\begin{equation}\label{eq:DefRoots}
G_n(z) = p_n^{(n)}\prod_{k=1}^{n}(z-r_k) = {2^n\over n!(n+1)!}\prod_{k=1}^{n}(z-r_k).
\end{equation}
However, this is precisely the probability generating function of a sum of $n$ independent Bernoulli random variables, see also \cite[Proposition 1]{Pitman}.

\begin{corollary}\label{cor:Bernoulli}
For each $n\in\NN$ let $B_1,\ldots,B_n$ be independent Bernoulli random variables satisfying 
$$
\PP(B_k=1) = {1\over 1-r_k}\qquad\text{and}\qquad \PP(B_k=0) = 1-{1\over 1-r_k},\qquad k = 1,\ldots,n,
$$
where $0=r_1,r_2,\ldots,r_n$ are the roots of $G_n(z)$ given by \eqref{eq:DefRoots}. Then the distributional identity $f_0(T_n) \overset{d}{=} 1 + B_2 + \ldots + B_n$ holds.
\end{corollary}

Note that since $r_1=0$, $B_1$ is always the constant `random' variable $1$. Next, for $n=2$ we have that
$$
f_0(T_2) \overset{d}{=} 1 + B_2\qquad\text{with}\qquad\PP(B_2 = 1) = {1\over 3},
$$
since $G_2(z)$ has roots $r_1=0$ and $r_2=-2$, while for $n=3$ it holds that
$$
f_0(T_3) \overset{d}{=} 1 + B_2 + B_3\qquad\text{with}\qquad\PP(B_2 = 1) = {1\over 5-\sqrt{7}},\quad\PP(B_3 = 1) = {1\over 5+\sqrt{7}},
$$
which arises from the fact that the roots of $G_3(z)$ are $r_1=0$, $r_2=-4+\sqrt{7}$ and $r_3=-4-\sqrt{7}$. Moreover, for $n=4$ we may define $\Delta:={1\over 3}(\arctan({27\sqrt{1895}\over 1142})-\pi)$ in order to express the roots of $G_4(z)$ as $r_1=0$, $r_2={1\over 3}({\sqrt{139}}\cos\Delta-{20})$, $r_3=-{1\over 3}(\sqrt{139}(\sqrt{3}\sin\Delta+\cos\Delta)+20)$ and $r_4=-{1\over 3}(\sqrt{139}(-\sqrt{3}\sin\Delta+\cos\Delta)+20)$. This leads to a  representation of $f_0(T_4)$ as a sum of independent Bernoulli random variables $1+B_2+B_3+B_4$ with success probabilities $\PP(B_k=1)=1/(1-r_k)$ for $k=2,3,4$.

\medbreak

Although the representation of $f_0(T_n)$ is not fully explicit in view of the unknown zeros $r_2,\ldots,r_n$, Corollary \ref{cor:Bernoulli} has a number of striking consequence as we shall demonstrate in the next paragraphs.

\paragraph{Mod-Gaussian convergence and bound on cumulants.}

The concept of mod-$\phi$ convergence is a recent, powerful and very elegant  tool to show probabilistic limit theorems. Namely, once mod-$\phi$ convergence has been established for a sequence of random variables, one automatically obtains a whole collection of probabilistic estimates and limit theorems, including the central limit theorem, the local limit theorem and moderate deviation bounds. We refer to \cite{FerayModPhi} for general background material.

In our case, we will have to deal with the special notion of so-called mod-Gaussian convergence only, for which we recall the following definition from \cite[Definition 1.1.1]{FerayModPhi}. Let $(Y_n)_{n\geq 1}$ be a sequence of random variables whose moment generating functions $\varphi_n(z):=\EE\exp(zY_n)$ are defined on some strip $S=\{z\in\CC:c_-<{\rm Re}(z)<c_+\}$ in the complex plane, where $-\infty\leq c_-<0<c_+\leq\infty$ (note that $S=\CC$ if $c_-=-\infty$ and $c_+=\infty$). The sequence $(Y_n)_{n\geq 1}$ convergences in the mod-Gaussian sense if
$$
\lim_{n\to\infty}{\varphi_n(z)\over e^{w_nz^2/2}} = \psi(z)
$$
locally uniformly on $S$, where $(w_n)_{n\geq 1}$ is some sequence of positive real numbers satisfying $w_n\to\infty$, as $n\to\infty$, and $\psi(z)$ is some analytic function on $S$. As explained in \cite{FerayModPhi}, mod-Gaussian convergence roughly means that $Y_n$ has approximately the same distribution as the $w_n$-th convolution power of the Gaussian distribution. The `difference' between these distributions is described by the so-called limit function $\psi(z)$, which plays a crucial role in the theory.

We show now mod-Gaussian convergence of the sequence of suitably normalized random variables $f_0(T_n)$. For that purpose let us recall from \cite[Corollary 2 and Corollary 3]{Buchta12} that for each $n\in\NN$ the variance $\sigma_n^2:=\var(f_0(T_n))$ and the third cumulant $L_n^3:=\kappa_3(f_0(T_n))$ of the vertex number $f_0(T_n)$ of the random convex chain $T_n$ are given by
\begin{align*}
	\sigma_n^2 &= {10\over 27}\sum_{k=1}^{n}{1\over n}+{4\over 9}\sum_{k=1}^{k}{1\over k^2}-{28\over 27}+{4\over 9(n+1)}
\end{align*}
and
\begin{align*}
	L_n^3 &= {14\over 81}\sum_{k=1}^{n}{1\over k}+{20\over 27}\sum_{k=1}^{n}{1\over k^2}+{16\over 27}\sum_{k=1}^{n}{1\over k^3}-{16\over 9}\sum_{k=1}^n\sum_{\ell=1}^{k}{1\over\ell^2} \\
	&\qquad\qquad\qquad\qquad+ {172\over 81}-{8\over 9}\Big({1\over n}+{1\over n+1}\Big)\sum_{k=1}^{n}{1\over k}-{28\over 27(n+1)},
\end{align*}
respectively. In particular, we note that
$$
\sigma_n^2=\Big({10\over 27}\log n\Big)(1+o(1))\qquad\text{and}\qquad L_n^3=\Big({14\over 81}\log n\Big)(1+o(1)),
$$
as $n\to\infty$, in the usual Landau notation.

\begin{corollary}\label{cor:modGauss}
The sequence of random variables
\begin{equation}\label{eq:defYn}
Y_n := {f_0(T_n)-\EE f_0(T_n)\over L_n},\qquad n\in\NN,
\end{equation}
converges in the mod-Gaussian sense in the whole complex plane with limiting function $\psi(z)=e^{z^3/6}$ and with $w_n={\sigma_n^2\over L_n^2}=\big({10\over 3}\big({3\over 14}\big)^{2/3}(\log n)^{1/3}\big)(1+o(1))$.
\end{corollary}
\begin{proof}
This is a direct consequence of \cite[Theorem 8.2.1]{FerayModPhi} in combination with Corollary \ref{cor:Bernoulli}. All assumptions in \cite{FerayModPhi} are readily checked from the explicit representation of $\sigma_n^2$ and $L_n^3$. In fact, using the notation in \cite{FerayModPhi} we have $k_n=n$, the polynomials $P_{n,j}$, $1\leq j\leq n$ are just the linear factors in \eqref{eq:DefRoots} and thus have maximal degree $1$, and $L_n=\Theta((\log n)^{1/3})=o((\log n)^{1/2})=o(\sigma_n)$, as $n\to\infty$.
\end{proof}

The random variables under consideration also satisfy sharp uniform cumulant bounds as considered in the limit theory for large deviations \cite{Saulis}. Let us recall that for $k\in\NN$ the $k$-th cumulant of a random variable $X$ is formally given by $\kappa_k(X):=(-i)^k{\dint^k\over\dint z^k}\log\EE\exp(izX)$. Combining now the representation of $f_0(T_n)$ as a sum of independent Bernoulli random variables from Corollary \ref{cor:Bernoulli} with \cite[Theorem 3.1]{Saulis} (where assumption $(B_\gamma)$ there is satisfied with $\gamma=0$ and $K=K_n=2$) one has that
\begin{equation}\label{eq:CumBound}
\kappa_k\Big({f_0(T_n)-\EE f_0(T_n)\over\sqrt{\var f_0(T_n)}}\Big) \leq k!\Big({4\over \var f_0(T_n)}\Big)^{k-2}
\end{equation}
for all integers $k\geq 3$.

\paragraph{Central limit theorem and normality zone.}

We continue our discussion about probabilistic properties of the random variables $f_0(T_n)$ with the central limit theorem in form of a Berry-Esseen bound. 

\begin{corollary}\label{cor:berryesseen}
For any $n\in\NN$ one has that
$$
\sup_{t\in\RR}\Big|\PP\Big({f_0(T_n)-\EE f_0(T_n)\over\sqrt{\var f_0(T_n)}}\geq t\Big)-\PP(Z\geq t)\Big| \leq {c\over\sqrt{\log n}}
$$
for some absolute constant $c\in(0,\infty)$, where $Z$ is a standard Gaussian random variable.
\end{corollary}
\begin{proof}
This is a consequence of the Berry-Esseen theorem for sums of independent but non-identically distributed random variables \cite{BatirovEtAl} and the fact that $f_0(T_n)$ can be represented as such a sum as we have seen in Corollary \ref{cor:Bernoulli}. Alternatively, it is also a direct consequence of the general Berry-Esseen bound \cite[Corollary 2.1]{Saulis} and the cumulant bound \eqref{eq:CumBound}.
\end{proof}

\begin{remark}\rm 
\cite[Corollary 2.1]{Saulis} in fact implies that the constant $c$ in the previous corollary can be taken as $c=72$.
\end{remark}

Corollary \ref{cor:berryesseen} implies that for any fixed $t\in\RR$,
\begin{equation}\label{eq:CLTequiv}
\lim_{n\to\infty}{\PP\Big({f_0(T_n)-\EE f_0(T_n)\over\sqrt{\var f_0(T_n)}}\geq t\Big)\over \PP(Z\geq t)} = 1.
\end{equation}
It is now natural to go one step further and to ask for the maximal scale $(s_n)_{n\geq 1}$ such that \eqref{eq:CLTequiv} continues to hold for $t=t_n$ depending on $n$ in such a way that $t_n=o(s_n)$. This maximal scale is the so-called normality zone, which in a sense, describes the `domain of attraction' of the Gaussian tail behaviour. In our set-up the following holds.

\begin{corollary}
The normality zone for the sequence of random variables ${f_0(T_n)-\EE f_0(T_n)\over\sqrt{\var f_0(T_n)}}$ is of size $o((\log n)^{1/6})$, that is, \eqref{eq:CLTequiv} holds if $t$ is replaced by a sequence $(t_n)_{n\geq 1}$ satisfying $t_n=o((\log n)^{1/6})$.
\end{corollary}
\begin{proof}
This is an immediate consequence of the mod-Gaussian convergence established in Corollary \ref{cor:modGauss} and \cite[Proposition 4.4.1]{FerayModPhi}.
\end{proof}

\paragraph{At the edge of the normality zone: moderate deviations.}

Our next result addresses the question of what happens at the edge of the normalizy zone. Here, the limiting function $\psi(z)=e^{z^3/6}$ shows up.

\begin{corollary}
For any $n\in\NN$ and $x>0$ one has that
\begin{align*}
\PP\Big({f_0(T_n)-\EE f_0(T_n)\over\sqrt{\var f_0(T_n)}}\geq \sqrt{t_n}\,x\Big) &= {e^{-t_n{x^2\over 2}}\over x\sqrt{2\pi t_n}}e^{x^3\over 6}(1+o(1)),\\
\PP\Big({f_0(T_n)-\EE f_0(T_n)\over\sqrt{\var f_0(T_n)}}\leq -\sqrt{t_n}\,x\Big) &= {e^{-t_n{x^2\over 2}}\over x\sqrt{2\pi t_n}}e^{-{x^3\over 6}}(1+o(1)),
\end{align*}
where $t_n={\sigma_n^2\over L_n^2}=\big({10\over 3}({3\over 14})^{2/3}(\log n)^{1/3}\big)(1+o(1))$.
\end{corollary}
\begin{proof}
This follows from the mod-Gaussian convergence established in Corollary \ref{cor:modGauss} and \cite[Theorem 4.2.1]{FerayModPhi}.
\end{proof}

By taking logarithms and the limit as $n\to\infty$ one obtains from the previous corollary that, for any $x>0$,
$$
\lim_{n\to\infty}{1\over t_n}\log\PP\Big({f_0(T_n)-\EE f_0(T_n)\over\sqrt{t_n\var f_0(T_n)}}\geq x\Big) = -{x^2\over 2}
$$
for the sequence $t_n={\sigma_n^2\over L_n^2}$. The moderate deviations principle asserts that this behaviour continues to hold for more general sequences $(t_n)_{n\geq 1}$ as well.

\begin{corollary}
Let $(t_n)_{n\geq 1}$ be a sequence of positive real numbers satisfying $t_n\to\infty$ and $t_n=o(\sigma_n^2)=o(\log n)$. Then, for any $x>0$,
$$
\lim_{n\to\infty}{1\over t_n}\log\PP\Big({f_0(T_n)-\EE f_0(T_n)\over\sqrt{t_n\var f_0(T_n)}}\geq x\Big) = -{x^2\over 2}.
$$
\end{corollary}
\begin{proof}
This follows immediately from the cummulant bound \eqref{eq:CumBound} together with \cite[Theorem 1.1]{DoeringEichelsbacher}.
\end{proof}

\begin{remark}\rm
The result in \cite{DoeringEichelsbacher} actually implies more than the limit relation in the previous corollary. We actually have that for any Borel set $B\subset\RR$ (and not just for $B$ of the form $[x,\infty)$ for $x>0$), 
\begin{align*}
-\inf_{x\in{\rm int}\,B}{x^2\over 2} &\leq\liminf_{n\to\infty}{1\over t_n}\log\PP\Big({f_0(T_n)-\EE f_0(T_n)\over\sqrt{t_n\var f_0(T_n)}}\in B\Big)\\
&\leq \limsup_{n\to\infty}{1\over t_n}\log\PP\Big({f_0(T_n)-\EE f_0(T_n)\over\sqrt{t_n\var f_0(T_n)}}\in B\Big)\leq -\inf_{x\in{\rm cl}\,B}{x^2\over 2},
\end{align*}
where ${\rm int}\,B$ and ${\rm cl}\, B$ denotes the interior and the closure of $B$, respectively and where $(t_n)_{n\geq 1}$ is any sequence of positive real numbers satisfying $t_n\to\infty$ and $t_n=o(\log n)$.
\end{remark}

\paragraph{The local limit theorem.}

Finally, we state the local limit theorem for $f_0(T_n)$. In a sense it describes how well the histogram of the normalized vertex number $f_0(T_n)$ can be approximated by that of the limiting Gaussian density.

\begin{corollary}
Let $\delta\in(0,1/2)$, $x\in\RR$ and $B\subset\RR$ be a Jordan measurable set with Jordan measure $J(B)>0$. Then
\begin{align*}
	\lim_{n\to\infty}{1\over(\log n)^\delta}\PP\Big({f_0(T_n)-\EE f_0(T_n)\over\sqrt{\var f_0(T_n)}} - x \in (\log n)^\delta B\Big) = {e^{-x^2/2}\over\sqrt{2\pi}}\,J(B).
\end{align*}
\end{corollary}
\begin{proof}
Again, this follows from the mod-Gaussian convergence established in Corollary \ref{cor:modGauss} and the general local limit theorem \cite[Theorem 3.2]{DalBorgoEtaL}.
\end{proof}

\subsection*{Acknowledgement}
This work has been supported by the Deutsche Forschungsgemeinschaft (DFG) via SPP 2265 \textit{Random Geometric Systems}.

\addcontentsline{toc}{section}{References}


\end{document}